\documentclass[pre,twocolumn,10pt]{revtex4}%
\usepackage{amsfonts}%
\usepackage{amsmath}%
\usepackage{amssymb}%
\usepackage{graphicx}
\newtheorem{theorem}{Theorem}

\newtheorem{lemma}[theorem]{Lemma}

\newenvironment{proof}[1][Proof]{\noindent\textbf{#1.} }{\ \rule{0.5em}{0.5em}}
\begin{document}
\title[Nondiscriminatory Propagation on Trees]{Nondiscriminatory Propagation on Trees}
\author{Simone Severini}
\affiliation{Institute for Quantum Computing and Department of Combinatorics \&
Optimization, University of Waterloo, Waterloo N2L 3G1, Canada}

\begin{abstract}
We consider a discrete-time dynamical process on graphs, firstly introduced in
connection with a protocol for controlling large networks of spin 1/2 quantum
mechanical particles [\emph{Phys. Rev. Lett.} \textbf{99}, 100501 (2007)]. A
description is as follows: each vertex of an initially selected set has a
packet of information (the same for every element of the set), which will be
distributed among vertices of the graph; a vertex $v$ can pass its packet to
an adjacent vertex $w$ only if $w$ is its only neighbour without the
information. By mean of examples, we describe some general properties, mainly
concerning homeomorphism, and redundant edges. We prove that the cardinality
of the smallest sets propagating the information in all vertices of a balanced
$m$-ary tree of depth $k$ is exactly $(m^{k+1}+\left(  -1\right)  ^{k}%
)/(m+1)$. For binary trees, this number is related to alternating sign matrices.

\end{abstract}
\maketitle

\section{Introduction}

\noindent\textbf{Background. }In view of applications like quantum RAM or
charge-coupled devices, Burgarth and Giovannetti \cite{bg} introduced a
protocol for arbitrarily control networks of coupled spin 1/2 quantum
particles (for example, an array of trapped ions). An important feature of the
protocol lies on the ability of transforming the physical state of the entire
network, by acting sequentially with the same local operation on a specific
subset of particles. This is valuable, since physical operations on quantum
objects are generally difficult to implement. It has been shown in \cite{bg}
that a network can be prepared in an arbitrary state by acting on the
particles of a subset, only if that subset satisfies certain conditions
related to the eigensystem of the Hamiltonian. Such conditions can be lifted
from the physical scenario and analyzed in a purely combinatorial setting.
This can be described in what follows as a discrete-time dynamical process on graphs.

\bigskip

\noindent\textbf{Definition. }Let $G=(V,E)$ be a simple undirected graph with
$V(G)=\{1,2,...,n\}$. Given a set $S\subset V(G)$, let $N[S]=\{w\in
V(G)\backslash S:\exists v\in S:\{v,w\}\in E(G)\}$ be the (closed)
neighborhood of $S$. Let $\mathcal{P}_{S}:\mathbb{[}n]\longrightarrow V(G)$ be
a map associating a subset of $V(G)$ to each \emph{time }$t\in\mathbb{[}%
n]=\{0,1,...,n-1\}$. We consider the following process:

\begin{itemize}
\item We select a set $S\subseteq V(G)$ and fix $\mathcal{P}_{S}(0)=S$.

\item For each $t\in\mathbb{[}n]\backslash\{0\}$, we have $\mathcal{P}%
_{S}(t)=\mathcal{P}_{S}(t-1)\cup T$, where $T\subseteq N[\mathcal{P}%
_{S}(t-1)]$. Moreover, if $w\in\mathcal{P}_{S}(t)\backslash\mathcal{P}%
_{S}(t-1)$ then there is $v\in\mathcal{P}_{S}(t-1)$ such that $\{v,w\}\in
E(G)$ and $N[v]\backslash\{w\}=\mathcal{P}_{S}(t-1)$.
\end{itemize}

In words, at time $t=0$, we select a subset of vertices $\mathcal{P}_{S}(0)$.
At time $t=1$, we may insert some vertices into $\mathcal{P}_{S}(0)$ and
obtain $\mathcal{P}_{S}(1)$. The \emph{propagation} will go on until
eventually $\mathcal{P}_{S}(k)=V(G)$, for some $k$. Clearly, $k\leq n-1$.
However, the propagation is not free, but it obeys some rules. Specifically, a
vertex $w$ can be inserted at time $t$ in $\mathcal{P}_{S}(t)$, only if it is
adjacent to $v\in\mathcal{P}_{S}(t-1)$ and all other neighbors of $v$ (except
$w$ itself) are already in $\mathcal{P}_{S}(t-1)$. If there is $k$ such that
$\mathcal{P}_{S}(k)=V(G)$, we say that $S$ \emph{propagates} to $G$. We denote
this fact by $S\mathcal{\curvearrowright}G$. Notice that $\mathcal{P}%
_{S}(t-1)\subseteq\mathcal{P}_{S}(t)$. We denote by $\#A$ the cardinality of a
generic set $A$. The figure below represents the steps of the described
dynamical process taking place on a small graph. The square vertices are the
elements of $\mathcal{P}_{S}(0)$. The grey vertices are the elements of
$V(G)\backslash\mathcal{P}_{S}(t)$. In the first row of the figure,
$\mathcal{P}_{S}(0)$ is on the left and $\mathcal{P}_{S}(1)$ on the right; in
the second row, $\mathcal{P}_{S}(2)$ and $\mathcal{P}_{S}(3)=V(G)$.%

\begin{center}
\includegraphics[
natheight=3.141000in,
natwidth=8.265900in,
height=0.8138in,
width=2.0946in
]%
{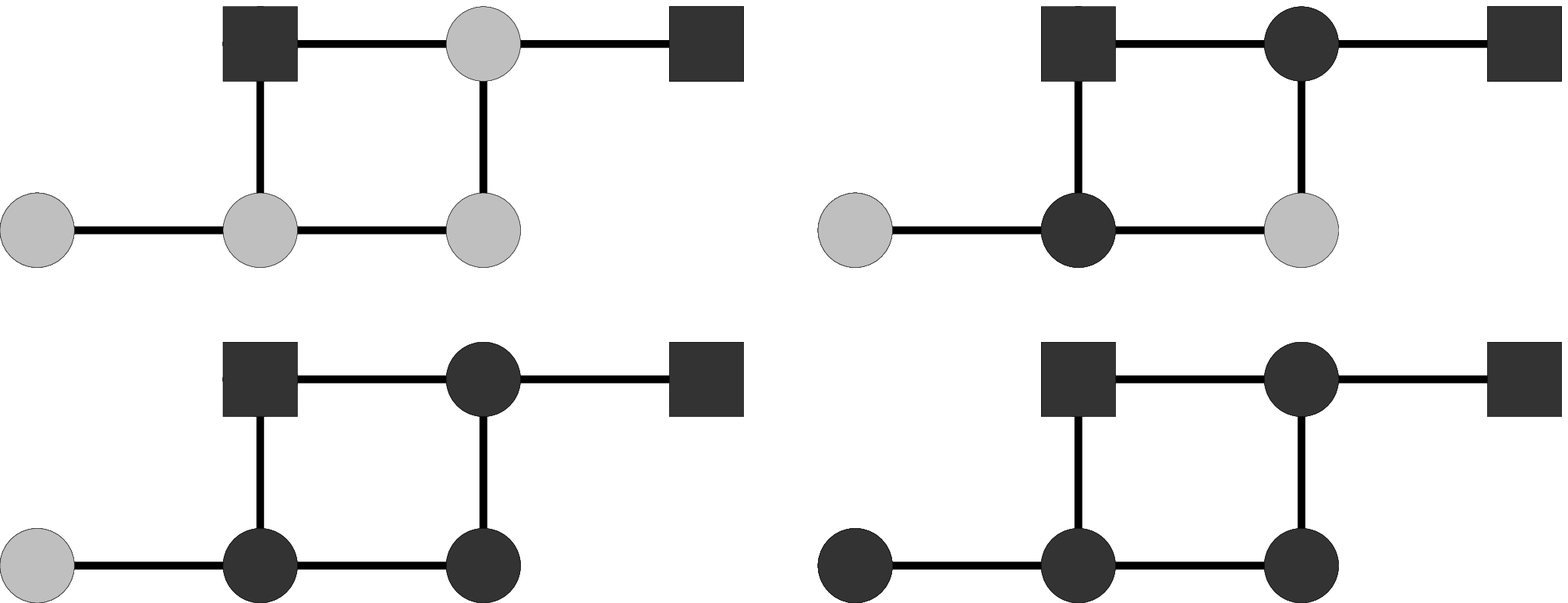}%
\end{center}

\noindent\textbf{Interpretation. }We may depict the above scenario in a more
concrete language: each vertex of an initially selected set has a packet of
information (the same for every vertex in the set), which has to be diffused
among the vertices of the graph; a vertex $v$ can pass its packet to an
adjacent vertex $w$ only if $w$ is its only neighbour still without the
information. In this way, $v$ does not need to \emph{discriminate} among its
neighbours, even if it is permitted to pass the information to only one of
those. Equivalently, this can be interpreted as a procedure for coloring (with
the same color) the vertices of a graph, in such a way that a vertex can be
colored at a certain time step, only if it is the unique uncolored neighbour
of an already colored vertex. Notice that the propagation is not
\emph{synchronized}, that is, we do not require that at a certain time $k$, a
vertex is necessarily included in $\mathcal{P}_{S}(k)$ if it is the unique
uncolored neighbor of a colored vertex.

\bigskip

\noindent\textbf{A quantitative question. }Here is a precise mathematical
problem: given a simple undirected graph $G$, find a set of minimum
cardinality that does propagate in $G$. The cardinality of such a set will be
denoted by $\pi(G)$. Obviously $\pi(G)$ is invariant under isomorphism, since
it does not depend on the labeling of the vertices. Notice that in this
problem we do not take into consideration the \emph{propagation time}. The
problem can be in fact modified by imposing time constraints. Finding $\pi(G)$
is of practical importance, when trying to optimize the number of local
operations required to initialize, and then control, a networks of spin $1/2$
particles. The computational complexity aspects of the question, a formulation
as an orientation problem, and approximation algorithms are studied in
\cite{a}. Roughly speaking, it looks like that $\pi(G)$ depends on the
expansion properties of $G$. Intuitively higher is the number of
\textquotedblleft ways out\textquotedblright\ from each subset of vertices of
a certain size and higher is $\pi(G)$.

\bigskip

\noindent\textbf{Structure of the paper. }We will focus mainly on trees. Apart
from the introduction, the paper contains two sections. In Section 2, we
underline some general properties of propagation, by taking as example paths,
combs and stars. We focus on homeomorphism and the maximum possible number of
edges that a graph can have, given the cardinality of the initial set
$\mathcal{P}_{S}(0)$. We finally make a comment about hamiltonicity and
propagation in digraphs. In Section 3, we will focus on balanced trees. The
main technical tool is a simple proof that the minimum cardinality of
$\mathcal{P}_{S}(0)$, such that $S$ propagates in a balanced binary tree,
realizes the Jacobsthal sequence. Our reference on the theory of graphs is the
book by Diestel \cite{de}.

\section{General facts by example}

\noindent\textbf{Homeomorphism. }It is worth keeping in mind that $\pi(G)$ is
invariant under homeomorphism. It is then more appropriate to think about
$\pi(G)$ not as a quantity associated to a single graph $G$, but rather to a
family of graphs, whose members are all the graphs homeomorphic to $G$. Recall
that graphs $G$ and $H$ are \emph{homeomorphic} if $H$ ($G$) can be obtained
by subdivision and smoothing on $G$ ($H$): a \emph{subdivision} of an edge
$\{u,v\}$ consists of deleting $\{u,v\}$, adding a vertex $w$, plus the edges
$\{u,w\}$ and $\{w,v\}$; a \emph{smoothing} is the reverse operation, and it
is then performed only on vertices of degree two. A \emph{tree} is a graph in
which any two vertices are connected by exactly one path. This property, plus
the homeomorphism remark, make propagation on trees particularly amenable to
quick observations.

\bigskip

\noindent\textbf{Paths. }Let $P_{n}$ be the path of length $n$. The path
$P_{n}$ models the classic spin chain with equal nearest neighbor couplings,
or more complex networks via the notion of graph covering and equitable
partitions (see, \emph{e.g.}, \cite{bo}). This is probably the graph for which
is simplest to determine $\pi(G)$. Indeed $\pi(P_{n})=1$, for every $n$. This
fact is self-evident and it does not need a proof. If $V(P_{n})=\{1,2,...,n\}$
and $\{v,w\}\in E(G)$ only if $w=v+1$, then it is sufficient to take $S=\{1\}$
or $S=\{n\}$. Clearly, $\mathcal{P}_{S}(n-1)=V(P_{n})$. Let $G_{n}=(V,E)$ be a
graph in which $\#V(G)=n$. Let $\mathcal{G}=\{G_{n}:G_{n}$ satisfies a
property $P$ for all $n\}$ be a family of graphs. We can look at the number
$\pi(G_{n})$ as a function $f:\mathbb{Z\longrightarrow}\mathbb{Z}$, defined as
$f(n)=\pi(G_{n})$, for all $G_{n}\in\mathcal{G}$. We have seen that $\pi
(P_{n})=1$ independently of $n$. This suggests the following structural graph
theory problem: characterize classes of graphs $\mathcal{G}$, for which
$\pi(G_{n})=c$, where $c$ is a constant, for all $G_{n}\in\mathcal{G}$. Paths
have this behaviour, since $\pi(P_{n})=1$, for all $n$. The same can be said
for $n$-cycles, since $\pi(C_{n})=2$. For complete graphs this is a linear
function: $\pi(K_{n})=n-1$. The next figure illustrates how
$\{1\}\mathcal{\curvearrowright}P_{4}$ (in the obvious way):%

\begin{center}
\includegraphics[
natheight=1.163200in,
natwidth=8.308300in,
height=0.3183in,
width=2.1049in
]%
{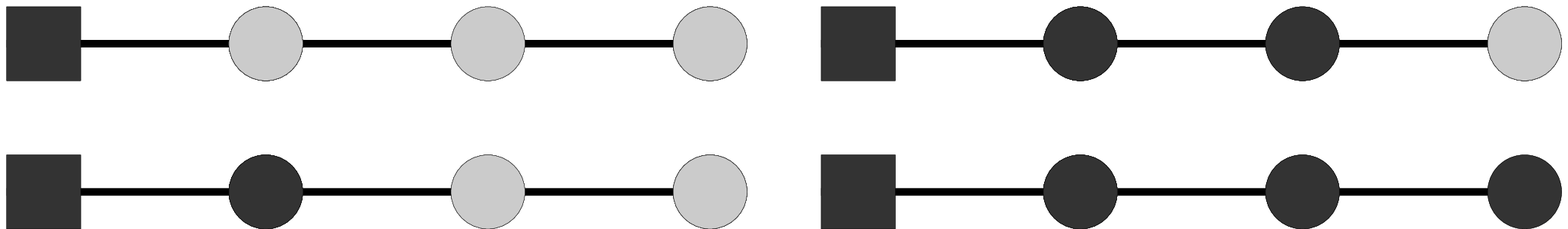}%
\end{center}

\noindent\textbf{Adding edges (I). }Since $\pi(P_{n})=1$ and the minimum
degree of $P_{n}$ is $1$, it is also natural to ask about graphs for which
$\pi(G)$ is exactly the minimum degree, that is, the trivial lower bound.
Paths suggest also another question: given a graph $G$ on $n$ vertices and
$S\subset V(G)$, what is the maximum number of edge that $G$ can have such
that $S\mathcal{\curvearrowright}G$. One can obtain a cycle $C_{n}$ from a
path $P_{n}$ by adding an extra edge to $P_{n}$. For covering $C_{n}$ by
propagation, we need $\#S\geq2$. Specifically, if $S$ contains just two
adjacent vertices then $S\mathcal{\curvearrowright}C_{n}$. Can we augment
$C_{n}$ by extra edges and keep $\#S=2$? Let $E(C_{n}%
)=\{1,2\},\{1,n\},\{2,3\},...,\{n-1,n\}$. Let $S=\{1,n\}$. Still
$S\mathcal{\curvearrowright}C_{n}+\{2,n\}$. More generally,
$S\mathcal{\curvearrowright}C_{n}+\bigcup_{i}\{2,i\}$. It is plausible to
conjecture that when $\#S=2$ and $V(G)=n$, then $\#E(G)=2n-3$ ($n\geq2$) is
the maximum possible number of edges that $G$ can have if
$S\mathcal{\curvearrowright}G$. The graph $C_{n}+\bigcup_{i}\{2,i\}$ attains
the bound. The graph $C_{5}+\{2,5\}\cup\{3,5\}$ is drawn below.%

\begin{center}
\includegraphics[
natheight=2.814100in,
natwidth=10.711600in,
height=0.7308in,
width=2.706in
]%
{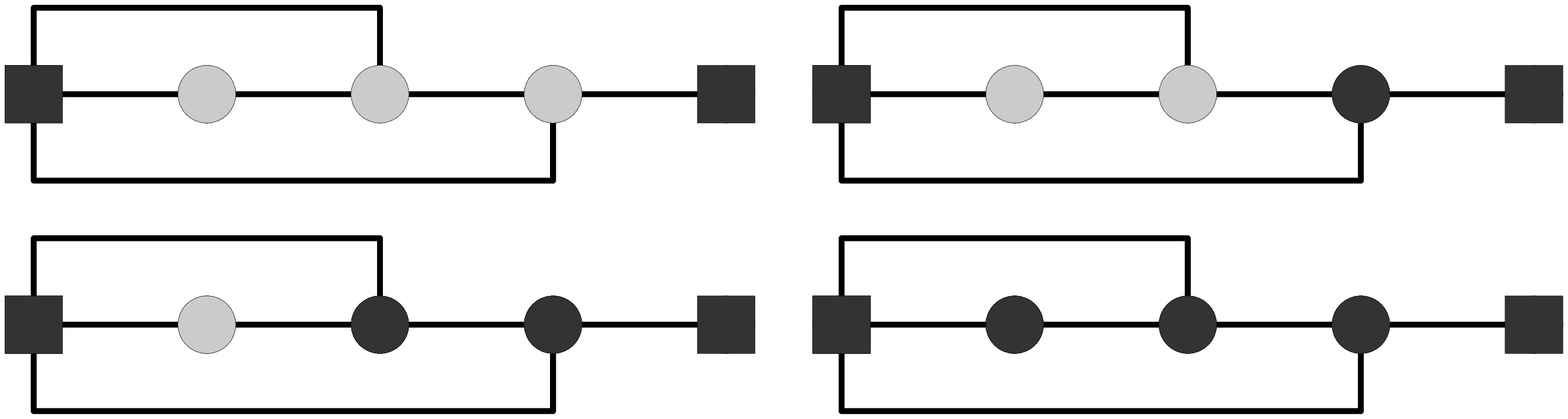}%
\end{center}

\noindent\textbf{Combs.} A \emph{comb} $P_{n,k}$ is a path $P_{n}$ having a
copy of $P_{k}$ attached to each vertex. Usually the plane embedding of this
tree is such that the copies of $P_{k}$ are all drawn in the upper region of
the plane determined by $P_{n}$. This justifies the term \textquotedblleft
comb\textquotedblright; the path $P_{n}$ is then called \emph{bone} and the
paths $P_{k}$ are called \emph{fingers}. So, $\#V(P_{n,k})=kn$. The comb
$P_{n,k}$ has $2$ vertices of degree $2$, $n-2$ vertices of degree $3$, and
$kn-n+2$ vertices and $k$ vertices of degree $1$. Given the invariance under
homeomorphism, it is sufficient to deal with $P_{n,2}$. In fact, longer
fingers attached to the vertices of the bone $P_{n}$ would not modify
$\pi(P_{n,k})$. Equivalently, $\pi(P_{n,k})=\pi(P_{n,2})$ for every $k$. We
have $\pi(P_{n,2})=n/2$ if $n$ is even and $\pi(P_{n,2})=\left\lceil
n/2\right\rceil $ if $n$ is odd. The figure shows how a $3$-element set
propagates in $P_{5,2}$.%

\begin{center}
\includegraphics[
natheight=3.135800in,
natwidth=10.669200in,
height=0.8121in,
width=2.6956in
]%
{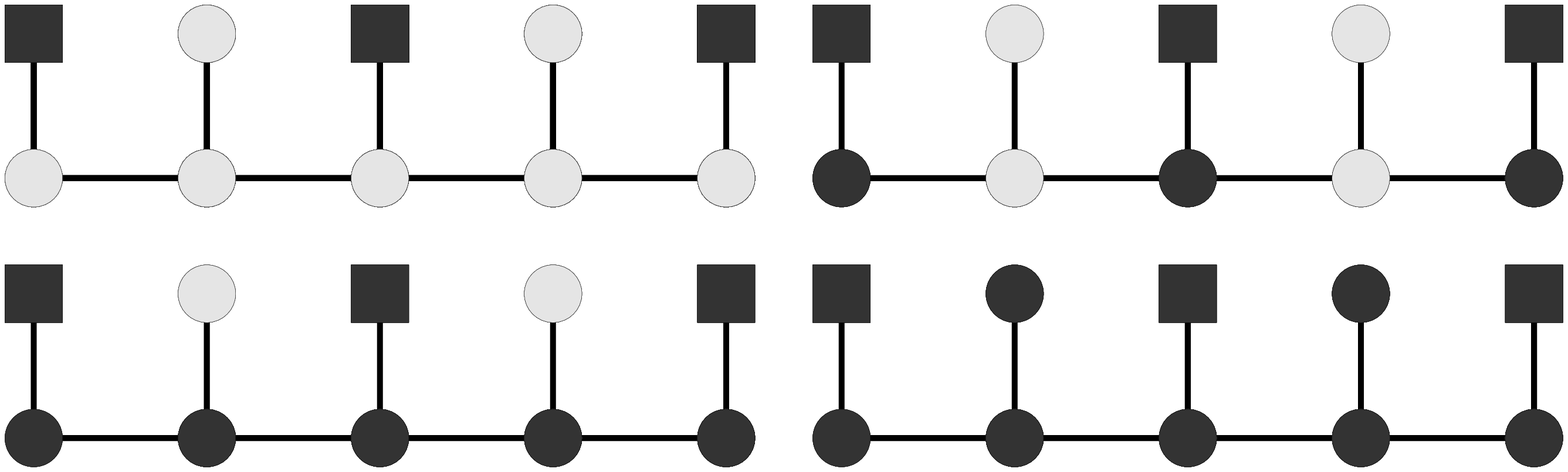}%
\end{center}

\noindent\textbf{Adding edges (II). }As we have already seen in the previous
paragraphs, in some situations one can add edges, and a pre-selected set will
still propagate in the graph. Certainly, one can always add edges connecting
only the pre-selected vertices and create a clique of size $\#S$. In combs,
differently from the case of $C_{n}$, we can construct a clique of size $n/2$,
containing then $\frac{1}{8}n^{2}-\frac{1}{4}$ edges. This implies that the
total number of edges is going to increase with a faster peace than in $C_{n}$
augmented by \emph{redundant} edges. In a generic graph, we can always add
redundant edges connecting vertices in $S$ and then complete the subgraph
induced by $S$ to a clique, without altering the dynamics. Other edges can be
added provided that these satisfy some conditions. Given $v\in\mathcal{P}%
_{S}(t)$, if we can add $\{v,w\}\in E(G)$ then:

\begin{enumerate}
\item Suppose $w\in\mathcal{P}_{S}(t)$. Then there is a vertex $y$ such that
$y\curvearrowright w$ at time $t^{\prime}>t$.

\item Suppose $w\notin\mathcal{P}_{S}(t)$.

\begin{enumerate}
\item If $N[v]\subseteq\mathcal{P}_{S}(t)$ then we can simply proceed to
include $w$ in $\mathcal{P}_{S}(t+1)$.

\item If $N[v]\nsubseteq\mathcal{P}_{S}(t)$ then there is a vertex $y$ such
that $y\curvearrowright w$ at time $t^{\prime}<t^{\prime\prime}$, where
$t^{\prime\prime}$ is the time step at which $v\curvearrowright z$, for some
vertex $z$.
\end{enumerate}
\end{enumerate}

We ask: what is the maximum number of edges in $P_{n,2}+H$ so that
$\pi(P_{n,2})=\pi(P_{n,2}+H)$? What about graphs in general? The answer is not
immediate and we leave it as an open problem.%

\begin{center}
\includegraphics[
natheight=2.258000in,
natwidth=9.012200in,
height=0.5924in,
width=2.2805in
]%
{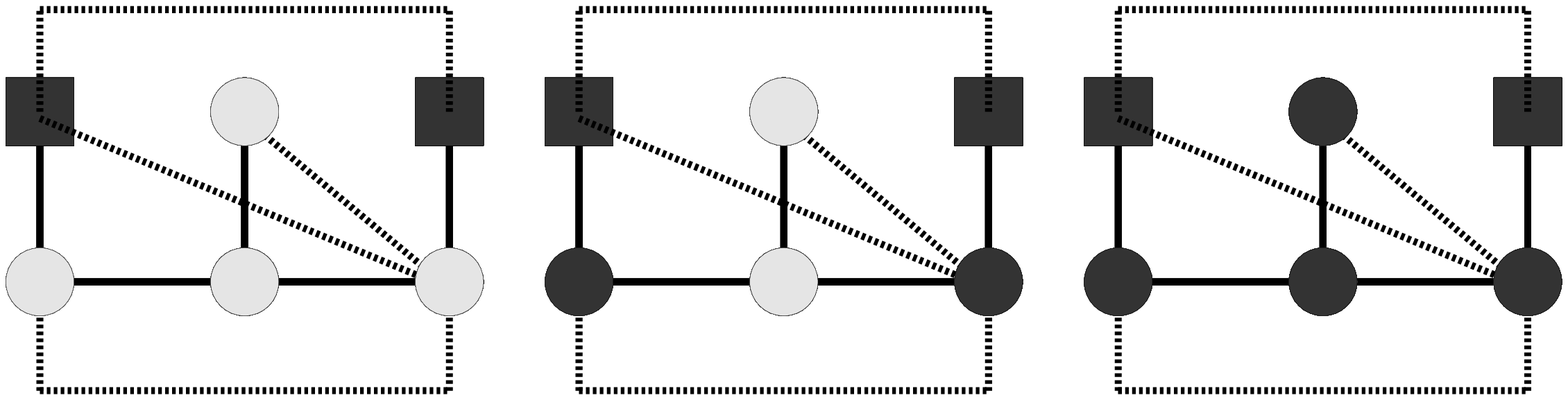}%
\end{center}

\noindent The figure shows propagation on the comb $P_{3,2}$, \emph{saturated}
with an additional number of edges. The extra edges are represented by dotted
lines. Note that adding a single vertex to a graph, in which we have already
fixed the initially selected vertices, may be sufficient to stop the propagation.

\bigskip

\noindent\textbf{Stars. }The complete bipartite graph $K_{1,n-1}$ is also said
to be a \emph{star} on $n$ vertices. Contextually to quantum networks,
properties of free bosons hopping on star networks where investigated in
\cite{ma}. If $S\curvearrowright K_{1,n-1}$ then $\#S=n-2$, by taking $n-2$
leaves (\emph{i.e.}, the vertices of degree 1). Among all graphs on $n$
vertices, the complete graph $K_{n}$ is the only graph for which the ratio
$n/\#\mathcal{P}_{S}(0)$ is smaller. If we include the root in $S$
(\emph{i.e.}, the vertex of degree $n-1$), then $\#S=n-1$. This implies that
we can add $\frac{1}{2}n^{2}-\frac{3}{2}n+1$ redundant edges to $K_{1,n-1}$
and obtain $K_{n}$ such that $S\curvearrowright K_{1,n-1}$ and
$S\curvearrowright K_{n}$, for exactly the same set $S$. If a graph $G$ has
$K_{1,n-1}$ as a spanning subgraph then $\#\mathcal{P}_{S}(0)\in\{n-2,n-1\}$,
for $G$. Recall that a \emph{spanning subgraph} is a subgraph that contains
all the vertices of the original graph. Valuable remarking that if a set
propagates in a graph then it will propagates in all of its spanning
subgraphs. Equally, the minimum cardinality of such a set is nonincreasing
when restricting ourselves to spanning subgraphs.

\bigskip

\noindent\textbf{Digraphs and hamiltonicity. }An \emph{orientation} of a graph
$G$ is a directed graph $\overrightarrow{G}$ obtained by giving a direction to
the edges of $G$ and in this way substituting $E(G)$ with a set of directed
\emph{arcs}. The propagation dynamics induces a partial ordering on the
vertices of $G$ and therefore an orientation. Observe that the definition
given in the introduction of this paper can be extended to digraphs in a
straightforward way. We ask: given a graph $G$, can we always find an
orientation $\overrightarrow{G}$ and a set $S$ such that $S\curvearrowright
\overrightarrow{G}$ and $\#S=1$. If $G$ has a Hamilton path then the answer is
in the affirmative, because we can just orient \emph{forward }the edges of the
Hamilton path and \emph{backwards} the remaining edges. Notice that the grid
considered in \cite{bg} is Hamiltonian. We can then always take a single
vertex to propagate in an arbitrary large grid, as far as we give a proper
orientation to the edges. Formally, let $\overrightarrow{P}_{n}$ be the
Hamilton directed path and let $S=\{1\}$. Within respect to $\overrightarrow
{G}$, we have $\mathcal{P}_{S}(t)=\{1,2,...,t\}$, for each $t$. The arc set of
$\overrightarrow{G}$ is $E(\overrightarrow{G})=E(\overrightarrow{P}_{n}%
)\cup\{(u,v):\{u,v\}\in E(G)\wedge v<u\}$. In the oriented version, we get
$\#\mathcal{P}_{S}(0)=1$ also for the complete graph. For $K_{n}$, the
orientation giving rise to a Hamilton directed path is obtained by
constructing any Hamiltonian tournament. Here is a picture of a single vertex
propagating in an orientation of $K_{4}$:%

\begin{center}
\includegraphics[
natheight=2.140400in,
natwidth=12.253500in,
height=0.563in,
width=3.0908in
]%
{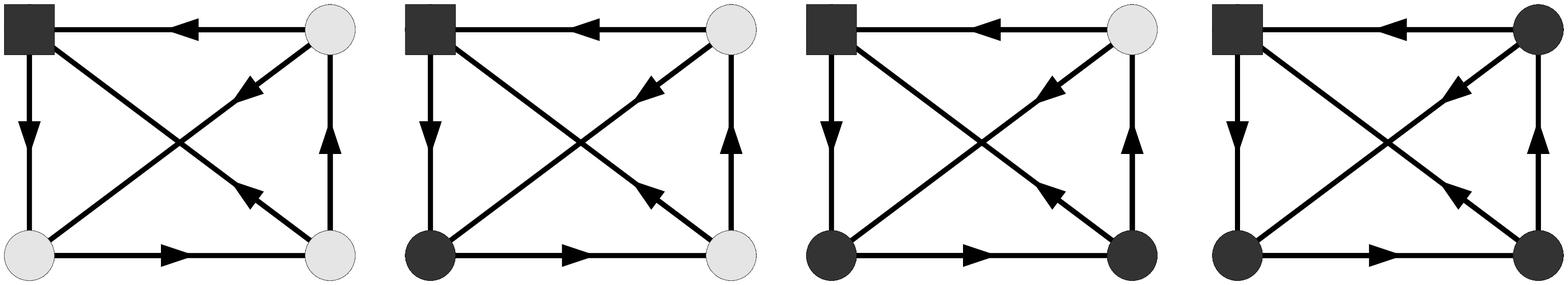}%
\end{center}

\section{Balanced trees}

Let us denote by $T_{2,k}$ a \emph{balanced binary tree of depth} $k$. Then
$\#V(T_{2,k})=2^{k}-1$. The root of $T_{2,k}$ is denoted by $v$. All the
remaining vertices are denoted by $v_{x}$, where $x\in\{0,1\}^{i}$, for
$i=1,...,k-1$. In particular, $\{v,v_{0}\},\{v,v_{1}\}\in E(T_{k,n})$ and
$\{v_{x},v_{x0}\},\{v_{x},v_{x1}\}\in E(T_{2,k})$, for every $x\in\{0,1\}^{i}
$ and $i=1,...,k-2$. The number of time steps needed to cover a tree is
necessarily equal to the diameter of the tree.

\bigskip

\noindent\textbf{Top-down propagation. }A set $S$ is said to propagate in
$T_{2,k}$ by \emph{topdown} propagation, when it propagates in $T_{2,k}$ and
if $v_{x}\in\mathcal{P}_{S}(t)$ then $x\in\{0,1\}^{i}$, with $i\leq t+1$, for
every $t$. Equivalently, in topdown propagation, the \emph{information flow}
goes from the root to the leafs. As a consequence, $v\in\mathcal{P}_{S}(0)$.
It is straightforward to determine $\pi_{TD}(T_{2,k})$, \emph{i.e.}, the
cardinality of the smallest set covering $T_{2,k}$ by topdown propagation:
given a tree $T_{2,k}$, we have $\pi_{TD}(T_{2,k})=2^{k-1}$. Let us see why.
If $S=\{v\}$ then $\mathcal{P}_{S}(1)=S$. So, we need to include $v_{0}$ or
$v_{1}$ in $S$. Let us take $S=\{v,v_{0}\}$. Now, $\mathcal{P}_{S}%
(1)=\{v,v_{0},v_{1}\}$. However, $\mathcal{P}_{S}(2)=\mathcal{P}_{S}(1)$. So,
we need to include $v_{00},v_{01},v_{10}$ or $v_{11}$ in $S$. In this way, if
$S=\{v,v_{0}\}\cup\{v_{x}:x=y0\wedge y\in\{0,1\}^{i},1\leq i\leq k-2\}$ then
$S\mathcal{\curvearrowright}G$. It is then clear that $\pi_{TD}(T_{2,k}%
)=\#S=2^{k-1}$.

\bigskip

\noindent\textbf{Bottom-up propagation. }A set $S$ is said to propagate in
$T_{2,k}$ by \emph{bottom-up} propagation, when it propagates in $T_{2,k}$ and
if $v_{x}\in\mathcal{P}_{S}(t)$ then $x\in\{0,1\}^{i}$, with $i=k-1$ if $t=0$,
$i=k-2$ if $t=1$, and so on. Equivalently, in bottom-up propagation, the
information flow goes from the leafs (\emph{i.e.}, the vertices of the form
$v_{x}$, with $x\in\{0,1\}^{k-1}$) to the root. It is obvious that $T_{2,k}$
is covered by bottom-up propagation if $\mathcal{P}_{S}(0)$ is the set of all
leaves, that is $\#\mathcal{P}_{S}(0)=2^{k-1}$. This is not equal to the
optimum, if we want that $v\in S$ only if $v$ is a leaf, without any further
constraint. The figure shows bottom-up propagation in $T_{2,4}$.%

\begin{center}
\includegraphics[
natheight=7.869800in,
natwidth=9.095200in,
height=1.9951in,
width=2.3021in
]%
{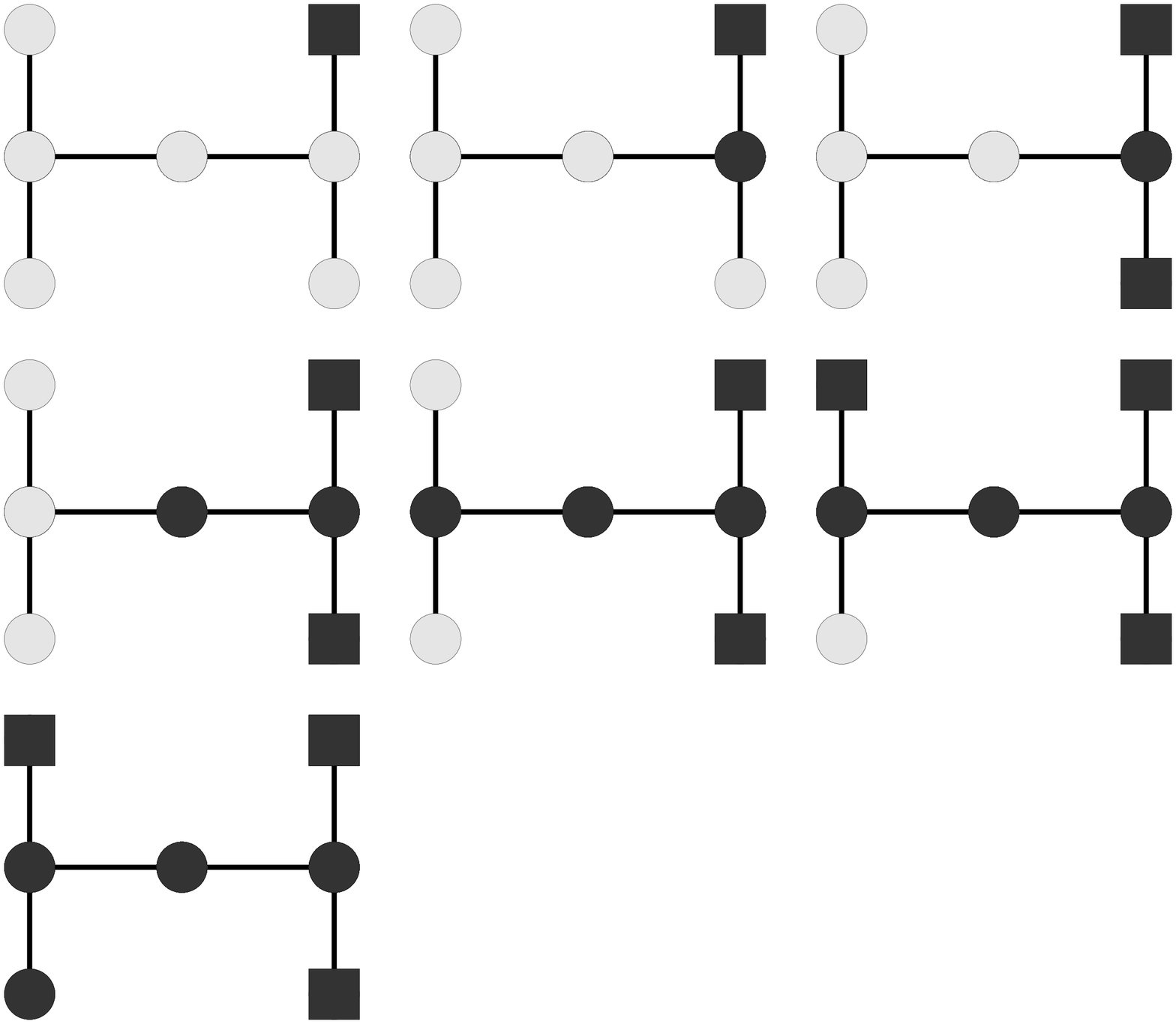}%
\end{center}

\noindent We will give a formal proof of the next result, a technical lemma
for establishing Theorem \ref{th}.

\begin{lemma}
\label{pro}For $T_{2,k}$ be a balanced binary tree of depth $k$. Then
\[
\pi_{leaf}(T_{2,k})=\pi(T_{2,k})=\frac{2^{k}+\left(  -1\right)  ^{k-1}}{3},
\]
the $(k-1)$-th Jacobsthal number.
\end{lemma}

\begin{proof}
First, take $T_{2,3}$. It is useful to write $\mathcal{P}_{S}^{k}(0)$ and
$S^{k}$ when considering $T_{2,k}$. Suppose the elements of $\mathcal{P}%
_{S}^{3}(0)$ being leaves only. We can start by including in $S^{3}$ a single
vertex, say $v_{00}$, in agreement with the notation defined. We will think of
the information flow going from \emph{bottom-left} to \emph{bottom-right},
with the propagation starts from $v_{00}$ and ending at $v_{11}$. (Just think
of $T_{2,3}$ drawn on the plane in the usual way.) We will add vertices in $S$
\emph{online}, as required, every time the propagation stops. In this way, we
provide that $\#S$ is as small as possible. We have $v_{00}\overset
{1}{\curvearrowright}v_{0}$ only if $v_{01}\in\mathcal{P}_{S}^{3}(0)$. The
notation is easy: $v_{00}$ propagates to $v_{0}$ at time $1$. Now
$S^{3}=\{v_{00},v_{01}\}$. So, $v_{0}\overset{2}{\curvearrowright}v$ and
$v\overset{3}{\curvearrowright}v_{1}$. At this stage, $v_{1}\overset
{4}{\curvearrowright}v_{11}$ only if $v_{10}\in S^{3}$. At the end
$S^{3}=\{v_{00},v_{01},v_{10}\}$ and $\#S^{3}=3$. It follows that $\pi
_{leaf}(T_{2,3})=\pi(T_{2,3})=3$. The tree $T_{2,4}$ can be constructed by
taking two copies of $T_{2,3}$ and adding an extra vertex adjacent of the
roots of the smaller trees. The new vertex is the root of $T_{2,4}$. We can
include the set $S^{3}$ in $\mathcal{P}_{S}^{4}(0)$ for $T_{2,4}$. The root of
$T_{2,4}$ is automatically covered and so the vertex $v_{1}\in V(T_{2,4})$.
This is sufficient to show that the set $S^{4}=\{v_{000},v_{001}%
,v_{100},v_{110}\}$ propagates in $T_{2,4}$. So, $\#\mathcal{P}_{S}%
^{4}(0)=\left(  2\cdot\#\mathcal{P}_{S}^{3}(0)\right)  -1=5$. Let $G$ be the
graph obtained by adding a pendant vertex to the root of $T_{2,4}$
(\emph{i.e.}, a vertex of degree one). For this graph $\pi(G)=\#\mathcal{P}%
_{S}^{4}(0)+1=6$ puts in evidence a recursive way to obtain $\pi
_{leaf}(T_{2,k})$. Since $T_{2,k+1}$ is constructed with two copies of
$T_{2,k}$ plus a new root, we can write $\#\mathcal{P}_{S}^{k}%
(0)=2(\#\mathcal{P}_{S}^{k-1}(0))+1$, for $k$ odd, and, $\#\mathcal{P}_{S}%
^{k}(0)=2(\#\mathcal{P}_{S}^{k-1}(0))-1$, for $k$ even. Such quantities are
exactly $\pi_{leaf}(T_{2,k})=\pi(T_{2,k})$, because the vertices inserted
online in $S^{k}$ are leaves. Odd and even cases are combined in the formula
$\pi(T_{2,k})=\frac{2^{k}+\left(  -1\right)  ^{k-1}}{3}$, the $(k-1)$-th
Jacobsthal number \cite{sl}.
\end{proof}

\bigskip

It has been pointed out that a set $S\curvearrowright G$ if $S$ represents a
configuration incompatible with a nontrivial eigenstate of the network
Hamiltonian \cite{bg}. Connections between the ground state vector for some
special spin systems and the alternating-sign matrices (ASMs) form an active
field of research in the interface between combinatorics, statistical
mechanics and condensed matter (see \cite{pr} and the references contained
therein). The number of ASMs of size $n$ is $A(n)=\prod_{l=0}^{n-1}%
\frac{(3l+1)!}{(n+l)!}$. Frey and Sellers \cite{f} proved that $A(n)$ is odd
if and only if $n$ is a Jacobsthal number. This observation could reveal a
potential link between ASMs and the physics (\emph{e.g.}, properties of the
eigensystem) involved in the protocol proposed in \cite{bg}, for networks
modeled by trees.

Theorem \ref{th} is the main point of this section. The proof is essentially
the same as the one of Lemma \ref{pro}. The sequences realized by $\pi
(T_{m,k})$ can be seen as generalizations of the Jacobsthal numbers.

\begin{theorem}
\label{th}Let $T_{m,k}$ be an balanced $m$-ary tree of depth $k$. Then
\[
\pi_{leaf}(T_{m,k})=\pi(T_{m,k})=\frac{m^{k+1}+\left(  -1\right)  ^{k}}{m+1}.
\]

\end{theorem}

The table below contains the first values of $\pi(T_{m,k})$:%

\[%
\begin{array}
[c]{ccccccc}%
m/k & \text{1} & \text{2} & \text{3} & \text{4} & \text{5} & \text{6}\\
\text{1} & \text{1} & \text{0} & \text{1} & \text{0} & \text{1} & \text{0}\\
\text{2} & \text{1} & \text{1} & \text{3} & \text{5} & \text{11} & \text{21}\\
\text{3} & \text{1} & \text{2} & \text{7} & \text{20} & \text{61} &
\text{182}\\
\text{4} & \text{1} & \text{3} & \text{13} & \text{51} & \text{205} &
\text{819}\\
\text{5} & \text{1} & \text{4} & \text{21} & \text{104} & \text{521} &
\text{2604}\\
\text{6} & \text{1} & \text{5} & \text{31} & \text{185} & \text{1111} &
\text{6665}%
\end{array}
\]
The number
\[
Q_{m,k}=\sum_{i=0}^{k+1}\left(  -1\right)  ^{k+1-i}m^{k+1-i}%
\]
is the $mk$-th entry of the table, disregarding of the signs. \emph{All}
sequences realized by the rows of the table appear to count walks of length
$k$ between any two vertices in the complete graph $K_{m+1}$, \emph{i.e.},
these are equal to $A(K_{m+1})_{i,j}^{k}$ (with $i\neq j$), where $A(K_{m})$
is the adjacency matrix of $K_{m}$. It is an open problem to exhibit a
bijection between each element in an initially selected set of minimal
cardinality propagating in $T_{m,k}$ and walks of length $k$ in $K_{m+1}$.

\bigskip

\noindent\emph{Acknowledgments. }The propagation model described here has been
suggested to me by Vittorio Giovannetti, during the \textquotedblleft The
first International Iran Conference on Quantum Information
(IICQI)\textquotedblright, held at Kish Island, in September 2007. I would
like to thank Ashkan Aazami and Joseph Cheriyan for their interest and
encouragement. I acknowledge the financial support of DTO-ARO, ORDCF,
Ontario-MRI, CFI, CIFAR, and MITACS.

\end{document}